\documentclass[11pt]{article}

\usepackage[margin=1in]{geometry}
\usepackage[T1]{fontenc}
\usepackage{lmodern}
\usepackage{microtype}
\usepackage{amsmath,amssymb,amsthm,mathtools}
\usepackage{array,booktabs,enumitem}
\usepackage{placeins}
\usepackage{tikz}
\usepackage{tikz-cd}
\usetikzlibrary{arrows.meta,positioning,calc}
\usepackage[hidelinks]{hyperref}
\hypersetup{
  pdftitle={Olshanskii Epimorphisms Are Standard},
  pdfauthor={Olga Kharlampovich and Alina Vdovina},
  pdfsubject={Standardness of finite-cover splitting epimorphisms},
  pdfkeywords={surface groups, free groups, splitting homomorphisms, Heegaard splittings, Seifert fibered spaces, Reidemeister-Schreier rewriting, Andrews-Curtis transformations}
}

\allowdisplaybreaks

\newtheorem{theorem}{Theorem}[section]
\newtheorem{proposition}[theorem]{Proposition}
\newtheorem{lemma}[theorem]{Lemma}

\theoremstyle{definition}
\newtheorem{definition}[theorem]{Definition}
\theoremstyle{remark}
\newtheorem{remark}[theorem]{Remark}

\newcommand{\Aut}{\operatorname{Aut}}
\newcommand{\rank}{\operatorname{rank}}

\title{Equations in Products of Free Groups and 3-Manifold Groups II: Olshanskii Epimorphisms}
\author{Olga Kharlampovich and  Alina Vdovina}
\date{}

\begin{document}
\maketitle

\begin{abstract}
In 1989 Olshanskii introduced a three-parameter family of
coordinate-surjective homomorphisms from the genus-two surface group to a
direct product of two rank-two free groups.  When the common quotient of
the two coordinate images is finite of order $n$, restriction to the
corresponding regular cover produces an epimorphism
\[
 \pi_1(S_{n+1})\longrightarrow F_{n+1}\times F_{n+1}.
\]
We call these maps \emph{Olshanskii epimorphisms}.  They form an explicit
high-genus test family for the standardness problem for splitting
epimorphisms.  We prove that every Olshanskii epimorphism is standard.
The genus-two homomorphism determines a Heegaard splitting of a Seifert
fibered $3$-manifold over $S^2$ with at most three exceptional fibers.  In
the finite-quotient cases, classical Seifert theory shows that its universal
cover is $S^3$; Waldhausen's theorem then implies that the lifted
genus-$(n+1)$ Heegaard splitting is standard.  The proof uses neither
Perelman's theorem nor the general Poincar\'e theorem.

We also give a constructive treatment of the quaternion case
$Q(2,2,2)\cong Q_8$, whose covering surface has genus nine.  A maximal tree
in the quaternion Schreier graph yields the covering handlebody and explicit
Schreier bases.  Using geometric longitude--meridian pairs and explicit
surface automorphisms supported on the nine one-holed tori, we transform the
lifted meridian words into a free basis.  A separate fixed-rank
Andrews--Curtis certificate reduces the associated balanced presentation.
This paper supplies the detailed proof of the result announced in
\cite{KharlampovichVdovina}.
\end{abstract}


\section{Introduction and motivation}

Let $S_g$ be a closed orientable surface of genus $g$ and let $F_g$ be a
free group of rank $g$.  A genus-$g$ Heegaard splitting
\[
 M=U_1\cup_{S_g}U_2
\]
of a closed orientable $3$-manifold determines two epimorphisms
$\pi_1(S_g)\twoheadrightarrow \pi_1(U_i)\cong F_g$, and hence a homomorphism
\[
 \pi_1(S_g)\longrightarrow F_g\times F_g.
\]
When the product map is onto, it is called a splitting epimorphism.  The
Stallings--Jaco--Hempel approach to the Poincar\'e conjecture asks whether
every such epimorphism is equivalent to the canonical one coming from the
standard splitting of $S^3$; see \cite{Stallings,Jaco,Hempel,GrigorchukKurchanov}.

Olshanskii's paper on homomorphism diagrams gave the first substantial
family designed to test this standardness problem \cite[Section~3.3]{Olshanskii}.
He starts with a three-parameter homomorphism
\[
 \alpha(k,\ell,m):\pi_1(S_2)\longrightarrow F_2\times F_2.
\]
The two coordinate maps are onto, but the product map usually is not.  Its
image is a fiber product over a common quotient $Q(k,\ell,m)$.  If this
quotient is finite of order $n$, then restriction to the corresponding
regular cover gives an epimorphism in genus $n+1$.  Thus a compact genus-two
construction produces complicated splitting epimorphisms in high genus.
For example, $(k,\ell,m)=(2,2,2)$ gives the quaternion group of order $8$
and a genus-nine epimorphism.  Olshanskii also records
\[
 Q(2,-3,5)\cong \mathrm{SL}(2,5)\times \mathbb Z_{31},
\]
of order $3720$, and hence an epimorphism in genus $3721$
\cite[Section~3.3]{Olshanskii}.

These examples are important for three related reasons.  First, they lie
directly in the standardness problem arising from Heegaard splittings.
Second, finite covering and Reidemeister--Schreier rewriting obscure the
standard handles: the resulting high-genus maps are far from the canonical
form in their natural generators.  Third, Olshanskii proposed a genuinely
diagrammatic recognition procedure, based on $a$-rings, $b$-rings, and the
successive excision of figure-eights.  He verified the quaternion case and
described a hand verification of the $(2,-3,5)$ case through intermediate
covers.  A uniform proof therefore tests whether standardness can be
recovered from the structure of the construction rather than from a
case-by-case inspection of very large diagrams.

The main result is the following.

\begin{theorem}[Main theorem]\label{thm:intro-main}
Every Olshanskii epimorphism is standard.
\end{theorem}

The proof is classical and geometric.  We identify the original marked
genus-two splitting with a Seifert fibered space over a spherical
$2$-orbifold.  Its universal cover is $S^3$ by the classical theory of
Seifert fibered spaces.  The lifted splitting is precisely the finite-cover
epimorphism constructed by Olshanskii, so Waldhausen's theorem makes it
standard.  No part of this argument invokes Perelman's theorem or the
general Poincar\'e theorem.

The companion paper \cite{KharlampovichVdovina} develops a broader
algebraic and geometric program for coordinate-surjective homomorphisms,
balanced presentations, and one-layer splitting epimorphisms.  It announces
the standardness of the Olshanskii epimorphisms.  The present paper isolates
that classical family, gives the complete marked-cover proof, and explains
how it is related to Olshanskii's original diagrammatic construction.  The
proof here is complementary to the one-layer methods of
\cite{KharlampovichVdovina}: the lifted maps need not be literal path
one-layer maps in their first Schreier bases.

The paper is organized as follows.  Section~\ref{sec:prelim} fixes the
notions of equivalence and standardness.  Section~\ref{sec:construction}
reconstructs Olshanskii's homomorphism and its finite-cover restriction.
Sections~\ref{sec:seifert} and \ref{sec:proof} identify the marked Seifert
splitting and prove the main theorem.  Section~\ref{sec:diagrams} compares
the argument with homomorphism diagrams and with the companion paper.
Section~\ref{sec:genus9} gives a fully explicit genus-nine construction.

\section{Splitting homomorphisms and standardness}\label{sec:prelim}

We use the conventions
\[
 [u,v]=u^{-1}v^{-1}uv,
 \qquad
 u^v=v^{-1}uv.
\]
Put
\[
 \Gamma_g=\pi_1(S_g)=
 \left\langle x_1,y_1,\ldots,x_g,y_g\ \middle|\
 [x_1,y_1]\cdots[x_g,y_g]=1\right\rangle.
\]

\begin{definition}
Two homomorphisms $\phi,\psi:\Gamma_g\to G$ are \emph{equivalent} if
there are $\sigma\in\Aut(\Gamma_g)$ and $\tau\in\Aut(G)$ such that
\[
 \phi=\tau\circ\psi\circ\sigma.
\]
A homomorphism $\Gamma_g\to F_g\times F_g$ is
\emph{coordinate-surjective} if both coordinate homomorphisms are onto.
\end{definition}

Let $F_A=F(a_1,\ldots,a_g)$ and $F_B=F(b_1,\ldots,b_g)$.  The canonical
splitting epimorphism is
\begin{equation}\label{eq:canonical}
 \varepsilon_g:\Gamma_g\longrightarrow F_A\times F_B,
 \qquad
 \varepsilon_g(x_i)=(a_i,1),
 \qquad
 \varepsilon_g(y_i)=(1,b_i).
\end{equation}
It is the splitting homomorphism of the standard genus-$g$ Heegaard
splitting of $S^3$.

\begin{definition}\label{def:standard}
An epimorphism $\phi:\Gamma_g\twoheadrightarrow F_g\times F_g$ is
\emph{standard} if it is equivalent to \eqref{eq:canonical}.
\end{definition}

This is Olshanskii's definition in \cite[Section~3.1]{Olshanskii}, up to a
possible interchange of the two free factors.  Hempel's reformulation of
the Poincar\'e conjecture says that every splitting epimorphism is standard
if and only if the Poincar\'e conjecture holds; see
\cite{Hempel,GrigorchukKurchanov}.

We shall use the following basis criterion, in the form used in the
companion paper \cite[Section~2]{KharlampovichVdovina}.

\begin{lemma}[Basis criterion for standardness]\label{lem:basis-criterion}
Let
\[
 \phi:\pi_1(S_g)\longrightarrow
 F(b_1,\ldots,b_g)\times F(a_1,\ldots,a_g)
\]
be a splitting epimorphism of the form considered here, with
\[
 \phi(x_i)=(b_i,a_i).
\]
Suppose that $U$ is a surface-group automorphism such that
\[
 \phi U(y_i)=(1,c_i),\qquad i=1,\ldots,g,
\]
and that
\[
 (c_1,\ldots,c_g)
\]
is a free basis of $F(a_1,\ldots,a_g)$.  Write the original generators
$a_i$ as words in this basis,
\[
 a_i=u_i(c_1,\ldots,c_g),
\]
and write
\[
 \phi U(x_i)=(B_i,C_i),
 \qquad
 B_i\in F(b_1,\ldots,b_g),\quad
 C_i\in F(c_1,\ldots,c_g).
\]
Then $\phi$ is standard.  More precisely, after further precomposition by
a surface-group automorphism fixing the $y_i$'s, and after automorphisms of
the two target free factors, the map has the form
\[
 x_i\longmapsto (b_i,c_i),
 \qquad
 y_i\longmapsto (1,c_i).
\]
\end{lemma}

\begin{proof}
Since $(c_1,\ldots,c_g)$ is a free basis, there is an automorphism
\[
 \alpha\in\Aut F(a_1,\ldots,a_g)
\]
such that
\[
 \alpha(c_i)=a_i.
\]

Let
\[
 \delta:\pi_1(S_g)\longrightarrow F(c_1,\ldots,c_g),
 \qquad
 \delta(x_i)=\delta(y_i)=c_i.
\]
By the description of the solutions of the surface equation used in the
paper \cite{KharlampovichVdovina}, there is a surface-group automorphism $W$ fixing every
$y_i$ such that
\[
 C_i=\delta W^{-1}(x_i),\qquad i=1,\ldots,g.
\]
Consequently, after precomposition by $W$,
\[
 \phi U W(x_i)=(B_i',c_i),
 \qquad
 \phi U W(y_i)=(1,c_i)
\]
for some $B_i'\in F(b_1,\ldots,b_g)$.

The tuple $(B_1',\ldots,B_g')$ generates $F(b_1,\ldots,b_g)$, because
$\phi U W$ is an epimorphism and all the $y_i$ have trivial first
coordinate.  Since this is a generating $g$-tuple of the rank-$g$ free
group, it is a free basis.  An automorphism of the first target factor
therefore sends $B_i'$ to $b_i$.  We obtain
\[
 x_i\longmapsto (b_i,c_i),
 \qquad
 y_i\longmapsto (1,c_i).
\]
Finally, precomposing by the product of the inverse twists about the
$y_i$'s changes $x_i$ to $y_i^{-1}x_i$ and fixes $y_i$.  The map then has
the form
\[
 x_i\longmapsto (b_i,1),
 \qquad
 y_i\longmapsto (1,c_i).
\]
Applying $\alpha$ to the second target factor gives the canonical splitting
epimorphism, up to the harmless interchange of the two target factors.
Hence $\phi$ is standard.
\end{proof}

\section{The Olshanskii construction}\label{sec:construction}

Fix nonzero integers $k,\ell,m$.  Let
\[
 F_1=F(a_1,a_2),
 \qquad
 F_2=F(b_1,b_2),
 \qquad
 T=(a_1a_2)^m,
 \qquad
 V=T a_2^{1-\ell}.
\]
Olshanskii's genus-two homomorphism is given by
\begin{equation}\label{eq:olsh-table}
\begin{array}{c|cccc}
 &x_1&y_1&x_2&y_2\\ \hline
F_1&
 a_1&(a_2a_1)^m a_1^{-k}&
 Va_2^{-1}V^{-1}&
 Va_2^{\ell}(a_1a_2)^{-m}V^{-1}\\[1mm]
F_2&b_1&1&b_2^{-1}&1.
\end{array}
\end{equation}
It is convenient to denote this map by $\alpha=\alpha(k,\ell,m)$ and its
coordinates by $\alpha_1,\alpha_2$.  The surface relation can be checked
directly; it will also follow from the marked Heegaard construction in
Proposition~\ref{prop:marked-splitting}.

\begin{lemma}\label{lem:coordinate-surjective}
Both coordinate homomorphisms in \eqref{eq:olsh-table} are onto.
\end{lemma}

\begin{proof}
The second coordinate is visibly onto.  In the first coordinate put
\[
 X=Va_2^{-1}V^{-1},
 \qquad
 Y=Va_2^{\ell}T^{-1}V^{-1}.
\]
Then
\[
 X^{\ell}Y=VT^{-1}V^{-1}
\]
and
\[
 (X^{\ell}Y)^{-1}X^{\ell-1}
 =VT a_2^{1-\ell}V^{-1}=V.
\]
Hence $V$ lies in the image, and therefore so does
\[
 a_2=V^{-1}X^{-1}V.
\]
The generator $a_1$ is already the image of $x_1$.
\end{proof}

Let
\[
 K=\alpha(\Gamma_2),
 \qquad
 N_1=K\cap(F_1\times\{1\}),
 \qquad
 N_2=K\cap(\{1\}\times F_2),
\]
and identify $N_i$ with a subgroup of $F_i$.  Since \(K\) projects onto both factors, the subgroups $N_1, N_2$
are normal in \(F_1\) and \(F_2\), respectively.  For \(u\in F_1\),
choose \(v\in F_2\) such that \((u,v)\in K\), and set
\[
\theta(u)=vN_2.
\]
This is well defined, surjective, and has kernel \(N_1\).  Hence it
induces an isomorphism
\[
F_1/N_1\cong F_2/N_2.
\]
Identifying these quotients with a common group \(Q\), and denoting the
quotient maps by \(q_i:F_i\to Q\), we obtain
\begin{equation}\label{eq:fiber-product}
 K=\{(u,v)\in F_1\times F_2\mid q_1(u)=q_2(v)\}.
\end{equation}
Equivalently,
\[
 Q\cong K/(N_1\times N_2)
   \cong F_1/N_1
   \cong F_2/N_2.
\]

The kernel of the standard handlebody epimorphism $\alpha_2$ is normally
generated by $y_1,y_2$.  Since $\alpha_1$ is onto, $N_1$ is therefore the
normal closure in $F_1$ of the two first-coordinate $y$-images.  Inverting
and conjugating the first one, and removing the whisker $V$ from the second,
shows that this normal closure is generated by
\[
 a_1^kT^{-1},
 \qquad
 a_2^{\ell}T^{-1}.
\]
Consequently
\begin{equation}\label{eq:Q}
 Q=Q(k,\ell,m)=
 \left\langle a_1,a_2\ \middle|\
 a_1^k=(a_1a_2)^m=a_2^{\ell}\right\rangle.
\end{equation}
The common value
\[
 h=a_1^k=a_2^{\ell}=(a_1a_2)^m
\]
is central.  Modulo $\langle h\rangle$ one obtains the triangle group
\[
 \left\langle \bar a_1,\bar a_2\ \middle|\
 \bar a_1^k=\bar a_2^{\ell}=(\bar a_1\bar a_2)^m=1\right\rangle.
\]
Olshanskii records that $Q(k,\ell,m)$ is finite precisely in the spherical
cases
\[
 \frac1{|k|}+\frac1{|\ell|}+\frac1{|m|}>1,
\]
and that these finite groups are cyclic groups or central extensions of
dihedral groups, $A_4$, $S_4$, and $A_5$
\cite[Section~3.3]{Olshanskii}.

Assume from now on that $Q$ is finite of order $n$.  The two compositions
$q_1\alpha_1$ and $q_2\alpha_2$ agree by \eqref{eq:fiber-product}; denote
their common value by
\[
 \rho:\Gamma_2\twoheadrightarrow Q.
\]
Put
\[
 H=\ker\rho=\alpha^{-1}(N_1\times N_2).
\]
Then $[\Gamma_2:H]=n$, so $H$ is the fundamental group of a regular
$n$-sheeted cover of $S_2$.  Since $\chi(S_2)=-2$, this cover has genus
$n+1$:
\[
 H\cong\pi_1(S_{n+1}).
\]
Also $[F_i:N_i]=n$, and the Nielsen--Schreier formula gives
\[
 \rank N_i=1+n(2-1)=n+1.
\]
Finally, $N_1\times N_2\subseteq K$, and $\alpha$ is onto $K$; hence the
restriction
\begin{equation}\label{eq:beta}
 \beta(k,\ell,m)=\alpha|_H:
 H\twoheadrightarrow N_1\times N_2
\end{equation}
is onto.

\begin{definition}
The epimorphisms \eqref{eq:beta}, for the triples for which $Q(k,\ell,m)$
is finite, are called \emph{Olshanskii epimorphisms}.
\end{definition}

With this terminology, Theorem~\ref{thm:intro-main} becomes the following
precise statement.

\begin{theorem}\label{thm:main}
Let $Q(k,\ell,m)$ be finite of order $n$.  After choosing free bases of
$N_1$ and $N_2$, the Olshanskii epimorphism
\[
 \beta(k,\ell,m):\pi_1(S_{n+1})\twoheadrightarrow
 F_{n+1}\times F_{n+1}
\]
is standard.
\end{theorem}

\section{The marked Seifert splitting}\label{sec:seifert}

Let $P$ be a pair of pants with oriented boundary curves
$c_1,c_2,c_3$ satisfying
\[
 c_1c_2c_3=1.
\]
Let $h$ denote the positive fiber of $P\times S^1$.  Fill the three boundary
tori of $P\times S^1$ so that the meridians of the filling solid tori are
\begin{equation}\label{eq:slopes}
 c_1^k h^{-1},
 \qquad
 c_2^{\ell}h^{-1},
 \qquad
 c_3^{-m}h^{-1}.
\end{equation}
Denote the resulting Seifert fibered space by $M_S(k,\ell,m)$.
Its fundamental group has the presentation
\begin{equation}\label{eq:seifert-presentation}
\begin{aligned}
\pi_1(M_S)=\langle c_1,c_2,c_3,h\mid{}&[c_i,h]=1,
 c_1c_2c_3=1,\\
&c_1^k=h,\quad c_2^{\ell}=h,\quad c_3^{-m}=h\rangle.
\end{aligned}
\end{equation}
Eliminating $c_3=(c_1c_2)^{-1}$ and $h$ gives \eqref{eq:Q}.  We need more
than an isomorphism of fundamental groups: we must identify the marked
Heegaard splitting.

\begin{figure}[t]
\centering
\begin{tikzpicture}[scale=0.92,>=Stealth]
  \draw[thick] (0,0) ellipse (3.25 and 2.25);
  \draw[thick] (-1.15,0.55) circle (0.66);
  \draw[thick] (1.15,0.55) circle (0.66);
  \coordinate (p) at (0,-2.12);
  \fill (p) circle (1.7pt);
  \draw[thick] (p) .. controls (-0.7,-1.25) and (-1.2,-0.65) .. (-1.15,-0.12);
  \draw[thick] (p) .. controls (0.7,-1.25) and (1.2,-0.65) .. (1.15,-0.12);
  \node at (-1.15,0.55) {$c_1$};
  \node at (1.15,0.55) {$c_2$};
  \node at (2.75,-1.45) {$c_3$};
  \node[left=2pt] at (-0.63,-0.92) {$d_1$};
  \node[right=2pt] at (0.63,-0.92) {$d_2$};
  \node[below=2pt] at (p) {basepoint};
  \draw[->] (-1.82,0.84) arc[start angle=155,end angle=465,radius=0.7];
  \draw[->] (0.48,0.84) arc[start angle=155,end angle=465,radius=0.7];
  \draw[->] (3.07,-0.65) arc[start angle=-15,end angle=40,x radius=3.25,y radius=2.25];
\end{tikzpicture}
\caption{The pair of pants $P$.  The arcs $d_1,d_2$ join $c_3$ to
$c_1,c_2$ and cut $P$ into a disk.  Their doubles give a complete disk
system in the genus-two open-book splitting.}
\label{fig:pants}
\end{figure}
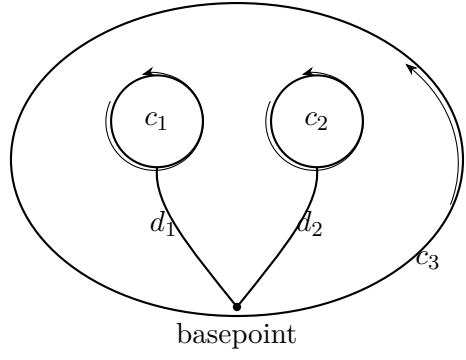

Let $\tau_i$ be a positive Dehn twist in a collar of $c_i$ and put
\begin{equation}\label{eq:monodromy}
 \varphi=\tau_1^{-k}\tau_2^{-\ell}\tau_3^{m}.
\end{equation}
We now describe the construction.  Start with $P\times[0,1]$, a one-parameter family of copies of
$P$.  Identify the last copy with the first by
\[
 (p,1)\sim(\varphi(p),0).
\]
The copies $P\times\{t\}$ are called the \emph{pages}, and $\varphi$ records
how the last page is attached back to the first.  The quotient is denoted
$T_\varphi$.  Since $P$ has three boundary circles, $T_\varphi$ has three
boundary tori.  On the torus corresponding to $c_i$, follow one point of
$c_i$ while $t$ runs once from $0$ to $1$; this gives a circle $\lambda_i$.
Attach a solid torus to each boundary torus so that $\lambda_i$ bounds a
meridian disk in the attached solid torus.  After the three attachments the
resulting $3$-manifold is closed.  This closed manifold is what is meant by
the open book with page $P$ and return map $\varphi$.

There is a simple genus-two splitting of this manifold.  Divide the
parameter circle into two closed arcs.  The pages over either arc, together
with the corresponding halves of the three attached solid tori, form a copy
of $P\times I$.  The disks $d_1\times I$ and $d_2\times I$ cut $P\times I$
into a $3$-ball, because $d_1,d_2$ cut $P$ into a disk.  Hence each half is
a genus-two handlebody.  Their common boundary consists of two copies of
$P$ joined along all three boundary circles.  This surface is the double
$D(P)$.  Since $\chi(P)=-1$, one has $\chi(D(P))=-2$, and therefore
\[
 D(P)\cong S_2,
\]
where $S_2$ denotes the closed orientable surface of genus two.

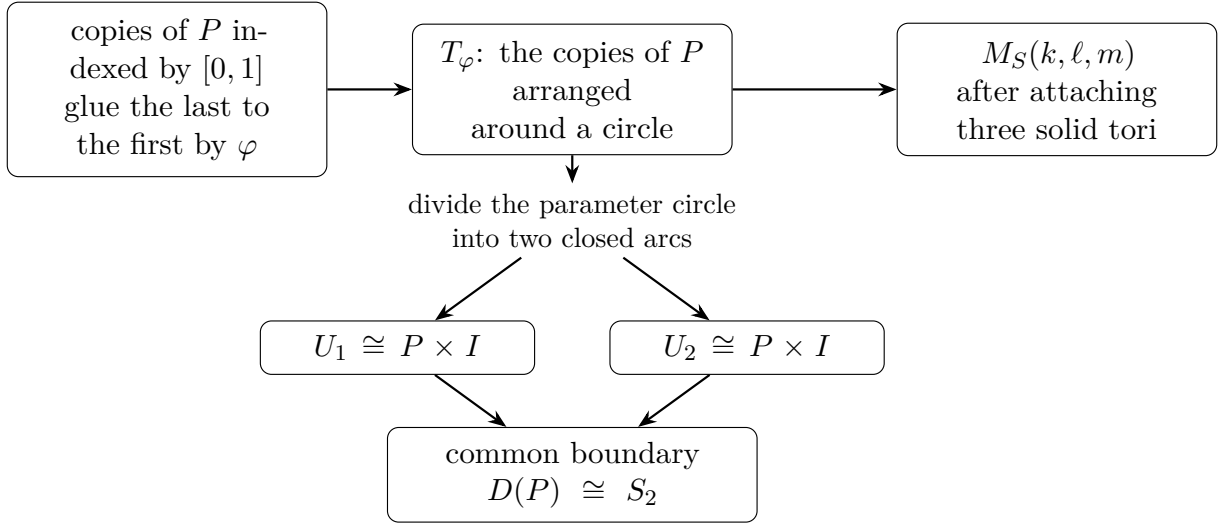
\begin{figure}[t]
\centering
\resizebox{0.97\textwidth}{!}{%
\begin{tikzpicture}[
  box/.style={draw,rounded corners,align=center,inner sep=6pt,text width=36mm},
  smallbox/.style={draw,rounded corners,align=center,inner sep=5pt,text width=31mm},
  >=Stealth]
  \node[box] (page) at (0,2.2) {copies of $P$ indexed by $[0,1]$\\
  glue the last to the first by $\varphi$};
  \node[box] (torus) at (5.1,2.2) {$T_\varphi$: the copies of $P$\\arranged around a circle};
  \node[box] (manifold) at (11.2,2.2) {$M_S(k,\ell,m)$\\after attaching three solid tori};
  \draw[->,thick] (page) -- (torus);
  \draw[->,thick] (torus) -- (manifold);

  \node[align=center,font=\small] (split) at (5.1,0.55)
  {divide the parameter circle\\into two closed arcs};
  \draw[->,thick] (torus) -- (split);

  \node[smallbox] (UA) at (2.9,-1.05) {$U_1\cong P\times I$};
  \node[smallbox] (UB) at (7.3,-1.05) {$U_2\cong P\times I$};
  \draw[->,thick] (split) -- (UA);
  \draw[->,thick] (split) -- (UB);

  \node[draw,rounded corners,align=center,inner sep=5pt,text width=43mm]
  (sigma) at (5.1,-2.65) {common boundary\\$D(P)\cong S_2$};
  \draw[->,thick] (UA) -- (sigma);
  \draw[->,thick] (UB) -- (sigma);
\end{tikzpicture}%
}
\caption{Building the closed $3$-manifold from copies of $P$, and
splitting it into two genus-two handlebodies.}
\label{fig:open-book}
\end{figure}

\begin{proposition}\label{prop:marked-splitting}
The open book \eqref{eq:monodromy} is the Seifert manifold
\(M_S(k,\ell,m)\).  Choose the identification of its Heegaard
surface with the standard surface \(S_2\), the canonical generators
\(x_1,y_1,x_2,y_2\), and the free bases
\[
\pi_1(U_1)=F(a_1,a_2),\qquad
\pi_1(U_2)=F(b_1,b_2)
\]
described below.  With these choices, the splitting homomorphism is
exactly the homomorphism \eqref{eq:olsh-table}.
\end{proposition}
\begin{proof}
Put
\[
 q_1=k,\qquad q_2=\ell,\qquad q_3=-m.
\]
Allowing the boundary points to move, choose a continuous family of maps
$f_t:P\to P$ from the identity to $\varphi$ such that a point of $c_i$
makes $-q_i$ turns around $c_i$ during the motion.  Then
\[
 [p,t]\longmapsto(f_t(p),e^{2\pi i t})
\]
gives a homeomorphism from $T_\varphi$ to $P\times S^1$.  Under this
homeomorphism the circle $\lambda_i$ described above represents
$c_i^{-q_i}h$; reversing its orientation gives $c_i^{q_i}h^{-1}$.  For
$i=1,2,3$ these are exactly the three curves in \eqref{eq:slopes}.
Therefore attaching the three solid tori in the open-book construction gives
$M_S(k,\ell,m)$.

We now read the marking.  Choose the arcs $d_1,d_2$ in
Figure~\ref{fig:pants}.  Their doubles bound disks in one of the two
handlebodies $P\times I$ described above.  Let
\[
 a_1=c_1,\qquad a_2=c_2,
 \qquad T=(a_1a_2)^m.
\]
The endpoint twists of $d_1,d_2$ give the two meridian words in the other
handlebody:
\begin{equation}\label{eq:mu}
 \mu_1=a_1^kT^{-1},
 \qquad
 \mu_2=a_2^{\ell}T^{-1}.
\end{equation}
This is the standard doubled-arc calculation: the endpoint on $c_i$
contributes $c_i^{q_i}$, while the common endpoint on $c_3$ contributes
$c_3^{-m}=T$ with the opposite boundary orientation.

Put
\begin{equation}\label{eq:delta}
 \delta=\mu_2^{-1}a_2=T a_2^{1-\ell}=V.
\end{equation}
Place the basepoint in the band joining the two once-punctured-torus parts
of $D(P)$.  On the first part take the longitude $a_1$ and the oppositely
oriented meridian $\mu_1^{-1}$; on the second take the longitude $a_2^{-1}$
and the meridian $\mu_2$.  The corresponding based canonical system is
\begin{equation}\label{eq:canonical-system}
 x_1=a_1,
 \qquad
 y_1=a_1^{-1}\mu_1^{-1}a_1,
 \qquad
 x_2=\delta a_2^{-1}\delta^{-1},
 \qquad
 y_2=\delta\mu_2\delta^{-1}.
\end{equation}
The two pairs intersect once in their respective punctured tori, are
otherwise disjoint, and have disk complement.  Thus they form a canonical
system on $S_2$.

In the first handlebody, equations \eqref{eq:mu}--\eqref{eq:canonical-system}
give
\[
\begin{aligned}
 x_1&=a_1,\\
 y_1&=a_1^{-1}Ta_1^{1-k}=(a_2a_1)^m a_1^{-k},\\
 x_2&=Va_2^{-1}V^{-1},\\
 y_2&=Va_2^{\ell}T^{-1}V^{-1}.
\end{aligned}
\]
In the second handlebody the curves $\mu_1,\mu_2$ bound disks.  If
$b_1,b_2$ are the dual core generators, then $a_1\mapsto b_1$,
$a_2\mapsto b_2$, and $\delta=\mu_2^{-1}a_2\mapsto b_2$.  Therefore
\[
 x_1\mapsto b_1,
 \qquad y_1\mapsto1,
 \qquad x_2\mapsto b_2^{-1},
 \qquad y_2\mapsto1.
\]
This is precisely \eqref{eq:olsh-table}.
\end{proof}

\begin{remark}[The role of $V$]\label{rem:V}
The word $V$ is not an additional topological block.  It is the whisker
from the common basepoint to the second punctured-torus pair.  It cannot be
simply deleted from Olshanskii's based table, because it occurs in both
$x_2$ and $y_2$.  It disappears from the unbased attaching curve and from
the common quotient because conjugating a relator changes only its
whisker.
\end{remark}

\begin{remark}[A fixed-diagonal representative]\label{rem:fixed-diagonal}
Changing the canonical system on the same marked splitting gives an
equivalent representative
\begin{equation}\label{eq:fixed-diagonal}
\begin{aligned}
 x_i&\longmapsto(a_i,b_i),\\
 y_1&\longmapsto\bigl(a_1^k(a_1a_2)^{-m},1\bigr),\\
 y_2&\longmapsto\bigl(a_2^{\ell}(a_2a_1)^{-m},1\bigr).
\end{aligned}
\end{equation}
The second $y$-word is a conjugate of the unbased meridian
$a_2^{\ell}(a_1a_2)^{-m}$.  We use \eqref{eq:fixed-diagonal} in the
quaternion calculation.  Equivalence carries the finite covering subgroup
to an isomorphic covering subgroup and therefore preserves the equivalence
class of the restricted epimorphism.
\end{remark}

\section{Proof of standardness}\label{sec:proof}

\begin{proposition}\label{prop:universal-cover}
If $Q(k,\ell,m)$ is finite, then the universal cover of
$M_S(k,\ell,m)$ is $S^3$.
\end{proposition}

\begin{proof}
Assume first that $|k|,|\ell|,|m|>1$.  The regular fiber represents the
central element
\[
 h=a_1^k=a_2^{\ell}=(a_1a_2)^m.
\]
Quotienting by $\langle h\rangle$ gives the finite spherical triangle group
\[
 Q/\langle h\rangle\cong
 \left\langle \bar a_1,\bar a_2\ \middle|\
 \bar a_1^k=\bar a_2^{\ell}=(\bar a_1\bar a_2)^m=1\right\rangle,
\]
which is the orbifold fundamental group of
$S^2(|k|,|\ell|,|m|)$.  Pull the Seifert fibration back to the universal
orbifold cover
\[
 S^2\longrightarrow S^2(|k|,|\ell|,|m|).
\]
The pullback is an ordinary oriented circle bundle
\[
 E\longrightarrow S^2
\]
which finitely covers $M_S$.  Since $\pi_1(E)$ is a subgroup of the finite
group $Q$, it is finite.  An oriented circle bundle over $S^2$ has
fundamental group $\mathbb Z/|e|\mathbb Z$ when its Euler number $e$ is
nonzero, and infinite cyclic fundamental group when $e=0$.  Hence here
$e\neq0$.  Such a bundle is the lens space $L(|e|,1)$, whose universal
cover is $S^3$.  The universal cover of $E$ is also the universal cover of
$M_S$, so $\widetilde M_S\cong S^3$.

If one of $|k|,|\ell|,|m|$ equals one, the Seifert fibration has at most two
genuine exceptional fibers.  The manifold is then a union of two solid
tori, hence a lens space unless it is $S^2\times S^1$.  The latter has
infinite fundamental group and is excluded by the finiteness of $Q$.
Again the universal cover is $S^3$.

This is the classical spherical part of Seifert theory
\cite{Seifert,Orlik}; it does not use the Poincar\'e theorem.
\end{proof}

\begin{figure}[t]
\centering
\begin{tikzcd}[column sep=large,row sep=large]
\widetilde U_1 \arrow[r,hook] \arrow[d,"p_1"'] &
S^3=\widetilde M_S \arrow[d,"p"] &
\widetilde U_2 \arrow[l,hook'] \arrow[d,"p_2"] \\
U_1 \arrow[r,hook] &
M_S(k,\ell,m) &
U_2 \arrow[l,hook']
\end{tikzcd}
\[
 \widetilde S=\partial\widetilde U_1=\partial\widetilde U_2,
 \qquad
 S_2=\partial U_1=\partial U_2.
\]
\caption{The universal cover of the marked genus-two splitting.  Each
preimage handlebody is connected because its fundamental group maps onto
$Q(k,\ell,m)$.}
\label{fig:lift}
\end{figure}

\begin{proof}[Proof of Theorem~\ref{thm:main}]
By Proposition~\ref{prop:marked-splitting}, the two coordinate maps of
$\alpha(k,\ell,m)$ are the inclusion homomorphisms of the marked genus-two
splitting
\[
 M_S=U_1\cup_{S_2}U_2.
\]
Let
\[
 p:S^3=\widetilde M_S\longrightarrow M_S
\]
be the universal cover.  Its degree is $n=|Q|$.  The image of
$\pi_1(U_i)=F_i$ in $\pi_1(M_S)=Q$ is all of $Q$, so
$\widetilde U_i=p^{-1}(U_i)$ is connected.  The corresponding subgroup is
\[
 \pi_1(\widetilde U_i)=\ker(F_i\to Q)=N_i.
\]
A connected $n$-sheeted cover of a genus-two handlebody is a handlebody of
genus
\[
 1+n(2-1)=n+1.
\]
Thus
\[
 S^3=\widetilde U_1\cup_{\widetilde S}\widetilde U_2,
 \qquad
 \widetilde S=p^{-1}(S_2)\cong S_{n+1},
\]
is a genus-$(n+1)$ Heegaard splitting; see Figure~\ref{fig:lift}.

The subgroup represented by the covering surface is
\[
 \pi_1(\widetilde S)
 =\ker(\Gamma_2\to Q)
 =\alpha^{-1}(N_1\times N_2)=H.
\]
Under the identifications
\[
 \pi_1(\widetilde S)=H,
 \qquad
 \pi_1(\widetilde U_i)=N_i,
\]
the splitting homomorphism of the lifted Heegaard splitting is exactly
\[
 H\longrightarrow N_1\times N_2,
 \qquad
 u\longmapsto\alpha(u),
\]
namely the Olshanskii epimorphism $\beta$.

Waldhausen's theorem says that every Heegaard splitting of $S^3$ is
isotopic to the standard splitting \cite{Waldhausen}.  Therefore the lifted
splitting is standard.  Passing to fundamental groups, there are
\[
 \sigma\in\Aut(H),
 \qquad
 \tau_i\in\Aut(N_i)
\]
such that
\[
 (\tau_1\times\tau_2)\,\beta\,\sigma
\]
is the canonical genus-$(n+1)$ epimorphism.  Hence $\beta$ is standard.
\end{proof}

\begin{remark}
The proof uses the explicit spherical classification of this Seifert
fibered family and Waldhausen's 1968 theorem.  It neither assumes that every
simply connected closed $3$-manifold is $S^3$ nor invokes geometrization.
The identification of the universal cover with $S^3$ is made directly by
passing to a circle bundle over $S^2$.
\end{remark}

\section{Homomorphism diagrams and paper \cite{KharlampovichVdovina} }\label{sec:diagrams}

Olshanskii's proof strategy is phrased in terms of homomorphism diagrams.
For an epimorphism
\[
 \pi_1(S_g)\twoheadrightarrow F_g\times F_g,
\]
he arranges the diagram into $g$ $a$-rings and $g$ $b$-rings.  When an
$a$-ring and a $b$-ring have exactly one common cell, their union is a
figure-eight, topologically a handle.  Cutting it off produces a lower-genus
diagram and suggests the corresponding Nielsen transformations in the two
free factors.  Repeating this operation verifies standardness
\cite[Sections~3.2--3.3]{Olshanskii}.

For the finite-cover family, the base diagram on $S_2$ lifts to the regular
cover determined by $Q(k,\ell,m)$.  The edge labels in the lift are obtained
by Reidemeister--Schreier rewriting in $N_1$ and $N_2$; zero-edges and
zero-cells may appear because rewriting can shorten labels.  Olshanskii
verified the quaternion example from the resulting development and used a
chain of intermediate covers for $(2,-3,5)$.  Theorem~\ref{thm:main}
replaces all such individual diagram calculations by one covering argument:
the entire lifted diagram is the Heegaard diagram of the universal cover of
a spherical Seifert manifold.

The companion paper \cite{KharlampovichVdovina} begins from a different
normal form.  After a change of canonical surface coordinates, the
unbased attaching relators of the genus-two parent are
\[
 a_1^k(a_1a_2)^{-m},
 \qquad
 a_2^{\ell}(a_1a_2)^{-m}.
\]
Thus the parent is a framed graph-layer construction supported on the
single edge joining the two handles.  The word $V$ has no role in the
associated balanced presentation because it is only a whisker, as explained
in Remark~\ref{rem:V}.  On the other hand, the finite-cover maps need not be
ordinary path one-layer maps in the Schreier bases naturally produced by
the cover.  The present Seifert proof therefore supplies the promised
standardness theorem without requiring a one-layer normal form upstairs.

\section{The quaternion Olshanskii epimorphism in genus nine}\label{sec:genus9}

We now make the case $(k,\ell,m)=(2,2,2)$ explicit.  This section is not
needed for the general proof of Theorem~\ref{thm:main}.  It gives a direct
standardization of the genus-nine map.  The essential point is that the
surface generators used below are chosen geometrically from the nine
handles of the covering handlebody; they are not arbitrary generators left
by a sequence of Tietze eliminations.

Work with the equivalent fixed-diagonal representative
\begin{equation}\label{eq:q8-alpha}
\begin{aligned}
 \widehat\alpha(x_i)&=(a_i,b_i),\\
 \widehat\alpha(y_1)&=\bigl(a_1^2(a_1a_2)^{-2},1\bigr),\\
 \widehat\alpha(y_2)&=\bigl(a_2^2(a_2a_1)^{-2},1\bigr).
\end{aligned}
\end{equation}
Let $q_a:F(a_1,a_2)\to Q$ and $q_b:F(b_1,b_2)\to Q$ be the two quotient
maps.  We use capital letters only for elements of the finite quotient:
\[
 A=q_a(a_1)=q_b(b_1),
 \qquad
 B=q_a(a_2)=q_b(b_2).
\]
Thus $a_1,a_2$ and $b_1,b_2$ remain free-group generators; they have not
been renamed as $A,B$.  The common quotient is
\begin{equation}\label{eq:q8-presentation}
 Q=\left\langle A,B\ \middle|\
 A^2=(AB)^2,\quad B^2=(BA)^2\right\rangle\cong Q_8.
\end{equation}
Indeed, $(BA)^2=A^{-1}(AB)^2A$, so the relations imply $A^2=B^2$,
$BAB^{-1}=A^{-1}$, and $A^4=1$.  Every element has a normal form
$A^rB^\epsilon$, $0\le r<4$, $\epsilon\in\{0,1\}$, and the usual
quaternion generators realize all eight forms.

Put
\[
 N_a=\ker q_a,
 \qquad
 N_b=\ker q_b,
 \qquad
 H=\ker\rho,
\]
where
\[
 \rho(x_1)=A,
 \quad \rho(x_2)=B,
 \quad \rho(y_1)=\rho(y_2)=1.
\]
Then
\[
 H\cong\pi_1(S_9),
 \qquad
 \rank N_a=\rank N_b=9,
\]
and the restriction
\[
 \beta=\widehat\alpha|_H:H\longrightarrow N_a\times N_b
\]
is the genus-nine Olshanskii epimorphism.

\subsection{The covering handlebody and a handle system adapted to a maximal tree}

Let $W_b$ be the standard genus-two handlebody with
\[
 \pi_1(W_b)=F(b_1,b_2),
\]
where $x_j$ is a longitude of its $j$th handle and $y_j$ is the
corresponding meridian.  Since $\rho(y_1)=\rho(y_2)=1$, the eight-sheeted
cover $S_9\to S_2=\partial W_b$ extends to the connected handlebody cover
\[
 \widetilde W_b\longrightarrow W_b
\]
associated with $N_b\leq F(b_1,b_2)$.  The handlebody $W_b$ retracts onto a
rose with two edges, and $\widetilde W_b$ retracts onto the corresponding
Schreier graph.

Choose the transversal
\[
 \mathcal T=\{1,A,A^2,A^3,B,AB,A^2B,A^3B\}.
\]
The Schreier graph has one vertex for each element of $Q_8$.  From a vertex
$q$ there is an $a_1$-edge to $qA$ and an $a_2$-edge to $qB$.  The same
graph, with edge labels $b_1,b_2$, is the spine of $\widetilde W_b$.
Consequently the capital labels on the vertices are quotient elements,
whereas the lower-case labels on the edges are free-group letters.

The phrase \emph{thicken the Schreier graph} means the following.  Replace
every vertex by a small $3$-ball and every edge by a tube joining the
corresponding balls.  Choose the maximal tree drawn with solid edges in
Figure~\ref{fig:q8-schreier}.  A regular neighborhood of this tree is a
$3$-ball.  There are sixteen positively oriented edges and eight vertices,
so nine edges lie outside the tree:
\[
 16-8+1=9.
\]
After the tree-neighborhood has been regarded as one ball, every non-tree
edge is one $1$-handle attached to that ball.  Hence the regular
neighborhood of the graph is a genus-nine handlebody.

For the $i$th non-tree edge, let $E_i$ be its co-core disk.  Choose a simple
closed curve $z_i$ on the boundary which passes once through this handle
and otherwise lies on the boundary of the tree-ball, and put
\[
 t_i=\partial E_i.
\]
The curves can be chosen and oriented so that
\[
 |z_i\cap t_j|=\delta_{ij},
 \qquad
 z_i\cap z_j=t_i\cap t_j=\varnothing\quad(i\ne j).
\]
Choose one basepoint outside the nine handle neighborhoods and a common
whisker to each pair.  Then
\[
 (z_1,t_1),\ldots,(z_9,t_9)
\]
is a symplectic surface system adapted to the chosen maximal tree, and
\[
 \pi_1(S_9)=
 \left\langle z_1,t_1,\ldots,z_9,t_9\ \middle|\
 [z_1,t_1]\cdots[z_9,t_9]=1\right\rangle.
\]

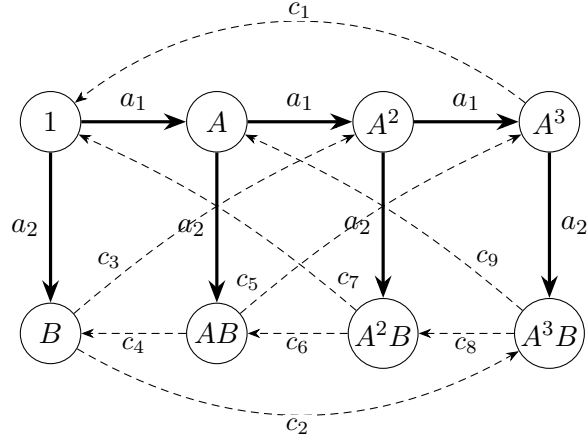
\begin{figure}[!htbp]
\centering
\begin{tikzpicture}[>=Stealth,
  v/.style={circle,draw,minimum size=8mm,inner sep=1pt},
  tree/.style={->,very thick},
  nontree/.style={->,densely dashed},
  elab/.style={fill=white,inner sep=1.2pt,font=\small}]
  \node[v] (e) at (0,2.8) {$1$};
  \node[v] (a) at (2.2,2.8) {$A$};
  \node[v] (a2) at (4.4,2.8) {$A^2$};
  \node[v] (a3) at (6.6,2.8) {$A^3$};
  \node[v] (b) at (0,0) {$B$};
  \node[v] (ab) at (2.2,0) {$AB$};
  \node[v] (a2b) at (4.4,0) {$A^2B$};
  \node[v] (a3b) at (6.6,0) {$A^3B$};

  \draw[tree] (e) -- node[above] {$a_1$} (a);
  \draw[tree] (a) -- node[above] {$a_1$} (a2);
  \draw[tree] (a2) -- node[above] {$a_1$} (a3);
  \draw[tree] (e) -- node[left] {$a_2$} (b);
  \draw[tree] (a) -- node[left] {$a_2$} (ab);
  \draw[tree] (a2) -- node[left] {$a_2$} (a2b);
  \draw[tree] (a3) -- node[right] {$a_2$} (a3b);

  \draw[nontree] (a3) to[bend right=38] node[elab,above] {$c_1$} (e);
  \draw[nontree] (b) to[bend right=30] node[elab,below] {$c_2$} (a3b);
  \draw[nontree] (b) to[bend left=8]
       node[elab,pos=.18,above left] {$c_3$} (a2);
  \draw[nontree] (ab) -- node[elab,below] {$c_4$} (b);
  \draw[nontree] (ab) to[bend left=8]
       node[elab,pos=.08,above left] {$c_5$} (a3);
  \draw[nontree] (a2b) -- node[elab,below] {$c_6$} (ab);
  \draw[nontree] (a2b) to[bend right=8]
       node[elab,pos=.08,above right] {$c_7$} (e);
  \draw[nontree] (a3b) -- node[elab,below] {$c_8$} (a2b);
  \draw[nontree] (a3b) to[bend right=8]
       node[elab,pos=.18,above right] {$c_9$} (a);
\end{tikzpicture}
\caption{The positive-edge Schreier graph of $Q_8$.  The capital letters
label quotient elements: $A=q_a(a_1)$ and $B=q_a(a_2)$.  The lower-case
letters label free-group edges.  The seven solid edges form a maximal tree.
The nine dashed edges become the nine handles of the covering handlebody
and give the Schreier basis $c_1,\ldots,c_9$ of $N_a$.}
\label{fig:q8-schreier}
\end{figure}
\FloatBarrier

\subsection{Schreier bases and the images of the handle curves}

For a non-tree edge $e_i:v_i\xrightarrow{a_{\epsilon_i}}w_i$, let $p_v$
be the tree path from the base vertex to $v$.  The loop
\[
 p_{v_i}e_ip_{w_i}^{-1}
\]
is the corresponding fundamental cycle.  Its label gives the following
Schreier basis of $N_a$:
\begin{equation}\label{eq:c-basis}
\begin{aligned}
 c_1&=a_1^4,&
 c_2&=a_2a_1a_2^{-1}a_1^{-3},&
 c_3&=a_2^2a_1^{-2},\\
 c_4&=a_1a_2a_1a_2^{-1},&
 c_5&=a_1a_2^2a_1^{-3},&
 c_6&=a_1^2a_2a_1a_2^{-1}a_1^{-1},\\
 c_7&=a_1^2a_2^2,&
 c_8&=a_1^3a_2a_1a_2^{-1}a_1^{-2},&
 c_9&=a_1^3a_2^2a_1^{-1}.
\end{aligned}
\end{equation}
Let $d_i$ be the corresponding basis of $N_b$, obtained by replacing
$a_j$ by $b_j$.  Since the longitude $z_i$ projects to the same fundamental
cycle in the two coordinate handlebodies,
\begin{equation}\label{eq:z-images}
 \beta(z_i)=(c_i,d_i).
\end{equation}

The co-core disk of an edge labelled $b_{\epsilon_i}$ projects to the
corresponding meridian disk of $W_b$.  Therefore, after using the tree path
to base its boundary, $t_i$ projects to
\[
 \tau_i(x_1,x_2)y_{\epsilon_i}\tau_i(x_1,x_2)^{-1},
\]
where
\[
\begin{array}{c|ccccccccc}
 i&1&2&3&4&5&6&7&8&9\\ \hline
 \tau_i&
 x_1^3&x_2&x_2&x_1x_2&x_1x_2&
 x_1^2x_2&x_1^2x_2&x_1^3x_2&x_1^3x_2\\
 \epsilon_i&1&1&2&1&2&1&2&1&2.
\end{array}
\]
Its second coordinate is trivial.  Reidemeister--Schreier rewriting of its
first coordinate gives
\begin{equation}\label{eq:r-images}
 \beta(t_i)=(r_i,1),
\end{equation}
where
\begin{equation}\label{eq:r-list}
\begin{aligned}
 r_1&=c_1c_9^{-1}c_2^{-1}c_1^{-1},&
 r_2&=c_2c_8c_9^{-1}c_2^{-1},&
 r_3&=c_3c_8^{-1}c_3^{-1},\\
 r_4&=c_4c_2c_3^{-1}c_4^{-1},&
 r_5&=c_5c_2^{-1}c_1^{-1}c_5^{-1},&
 r_6&=c_6c_4c_1^{-1}c_5^{-1}c_6^{-1},\\
 r_7&=c_7c_4^{-1}c_7^{-1},&
 r_8&=c_8c_6c_7^{-1}c_8^{-1},&
 r_9&=c_9c_6^{-1}c_9^{-1}.
\end{aligned}
\end{equation}
For example, the first non-tree edge is the $a_1$-edge from $A^3$ to
$1$.  Its fundamental cycle has label $a_1^4=c_1$, while the boundary of
its co-core projects to $x_1^3y_1x_1^{-3}$.  Rewriting the first coordinate
of this word gives
\[
 r_1=c_1c_9^{-1}c_2^{-1}c_1^{-1}.
\]

\subsection{Genuine surface automorphisms supported on the lifted handles}

For each $i$, let $N_i$ be a small regular neighborhood in $S_9$ of
$z_i\cup t_i$.  Since $z_i$ and $t_i$ meet once, $N_i$ is a one-holed
torus.  The neighborhoods $N_i$ may be chosen pairwise disjoint.  The next
lemma constructs the required source automorphism from actual Dehn twists.

\begin{lemma}\label{lem:local-D}
Let $z,t$ be the longitude and meridian of a one-holed torus $N$.  Choose
the orientations of the Dehn twists so that, on the based fundamental
group,
\[
 T_z(z)=z,\qquad T_z(t)=zt,
\]
and
\[
 T_t(z)=tz,\qquad T_t(t)=t.
\]
Compositions are read from right to left.  Put
\[
 C=T_z^{-1}T_tT_z^{-1},
 \qquad
 D=C^2.
\]
Then $D$ is induced by a homeomorphism supported in $N$ and equal to the
identity near $\partial N$, and
\begin{equation}\label{eq:local-D-action}
 D(t)=z^{-1}t^{-1}z,
 \qquad
 D(z)=z^{-1}t^{-1}z^{-1}tz.
\end{equation}
\end{lemma}

\begin{proof}
The maps $T_z$ and $T_t$ are Dehn twists about the two simple curves, so
they are homeomorphisms supported in $N$ and may be chosen to fix
$\partial N$ pointwise.  Direct calculation gives
\[
 C(t)=T_z^{-1}T_t(z^{-1}t)=z^{-1}
\]
and
\[
 C(z)=T_z^{-1}T_t(z)=z^{-1}tz.
\]
Consequently
\[
 D(t)=C(z^{-1})=z^{-1}t^{-1}z
\]
and
\[
 D(z)=C(z^{-1}tz)=z^{-1}t^{-1}z^{-1}tz.
\]
Thus $D$ is a product of genuine Dehn twists, not merely a formal
free-group transformation.  Since it fixes the boundary of $N$, it extends
by the identity to the whole surface.
\end{proof}

Apply Lemma~\ref{lem:local-D} in $N_i$ and denote the resulting surface
automorphism by $D_i$.  It fixes every handle outside $N_i$ and satisfies
\[
 D_i(t_i)=z_i^{-1}t_i^{-1}z_i.
\]
Since the supports are disjoint, the automorphisms $D_i$ commute.  Put
\[
 D=D_1\cdots D_9.
\]
The formula for $D_i(z_i)$ will not be needed below.  It is important only
that $D$ is a genuine automorphism of the surface group; the basis
criterion allows arbitrary resulting images of the longitude generators.

\subsection{The transformed meridian words form a free basis}

Using \eqref{eq:z-images}, \eqref{eq:r-images}, and the direct-product
multiplication, we obtain
\[
\begin{aligned}
 \beta D(t_i)
 &=\beta(z_i)^{-1}\beta(t_i)^{-1}\beta(z_i)\\
 &=(c_i^{-1}r_i^{-1}c_i,1).
\end{aligned}
\]
Put
\[
 q_i=c_i^{-1}r_i^{-1}c_i.
\]
Then \eqref{eq:r-list} gives
\begin{equation}\label{eq:q-list}
\begin{aligned}
 q_1&=c_2c_9,&
 q_2&=c_9c_8^{-1},&
 q_3&=c_8,\\
 q_4&=c_3c_2^{-1},&
 q_5&=c_1c_2,&
 q_6&=c_5c_1c_4^{-1},\\
 q_7&=c_4,&
 q_8&=c_7c_6^{-1},&
 q_9&=c_6.
\end{aligned}
\end{equation}
This tuple is a free basis of $N_a$.  Indeed, using the current value of
each word at every step, make the following Nielsen transformations:
\begin{equation}\label{eq:nielsen-sequence}
\begin{aligned}
 q_2&\leftarrow q_2q_3=c_9,\\
 q_1&\leftarrow q_1q_2^{-1}=c_2,\\
 q_4&\leftarrow q_4q_1=c_3,\\
 q_5&\leftarrow q_5q_1^{-1}=c_1,\\
 q_6&\leftarrow q_6q_7=c_5c_1,\\
 q_6&\leftarrow q_6q_5^{-1}=c_5,\\
 q_8&\leftarrow q_8q_9=c_7.
\end{aligned}
\end{equation}
The resulting ordered tuple is
\[
 (c_2,c_9,c_8,c_3,c_1,c_5,c_4,c_7,c_6),
\]
a permutation of the Schreier basis \eqref{eq:c-basis}.

\begin{proposition}\label{prop:q8-standard}
The genus-nine Olshanskii epimorphism is standard.
\end{proposition}

\begin{proof}
Interchange the two target factors and write
\[
 \phi:\pi_1(S_9)\longrightarrow N_b\times N_a.
\]
With the free bases $(d_1,\ldots,d_9)$ and $(c_1,\ldots,c_9)$, equations
\eqref{eq:z-images} and \eqref{eq:r-images} say
\[
 \phi(z_i)=(d_i,c_i),
 \qquad
 \phi(t_i)=(1,r_i).
\]
After precomposition by the genuine surface automorphism $D$,
\[
 \phi D(t_i)=(1,q_i).
\]
By \eqref{eq:nielsen-sequence}, the tuple $(q_1,\ldots,q_9)$ is a free
basis of $N_a$.  Apply Lemma~\ref{lem:basis-criterion} with
\[
 x_i=z_i,
 \qquad
 y_i=t_i,
 \qquad
 U=D.
\]
It follows that $\phi$, and hence $\beta$, is standard.
\end{proof}

\begin{remark}\label{rem:geometric-generators}
The proof uses the geometric longitudes and co-core meridians chosen before
any Tietze elimination.  This is why the neighborhoods $N_i$, the Dehn
twists $T_{z_i},T_{t_i}$, and the surface automorphisms $D_i$ are defined.
An arbitrary algebraic generating set obtained after
Reidemeister--Schreier rewriting and elimination need not already be paired
into such handles; it would first have to be identified with a geometric
surface system.
\end{remark}

\subsection{The associated Andrews--Curtis certificate}

The same calculation gives a fixed-rank certificate for the associated
balanced presentation.  This is a presentation-level statement; it is not
being used in place of the surface-automorphism argument above.

Passing from $r_i$ to
\[
 q_i=c_i^{-1}r_i^{-1}c_i
\]
uses inversion and conjugation of the individual relators.  The
transformations in \eqref{eq:nielsen-sequence} multiply one relator by
another or its inverse, and the final tuple is a permutation of the ambient
free basis.  No generator--relator pair is added or deleted.  Therefore
\[
 \left\langle c_1,\ldots,c_9\ \middle|\ r_1,\ldots,r_9\right\rangle
\]
is Andrews--Curtis equivalent at fixed rank to the standard trivial
presentation.

Olshanskii's Figure~6 is a development of the lifted homomorphism diagram
and contains twenty-one nonzero cells.  Figure~\ref{fig:q8-schreier}
records different information: it is the quotient graph that gives the
covering handlebody, the handle system adapted to the maximal tree, and the
two Schreier bases.

\section{Concluding remarks}

The finite-cover construction has two levels.  The genus-two parent map is
simple and its common quotient is a central extension of a spherical
triangle group.  The restricted map in genus $n+1$ is combinatorially much
larger, but it is not a new arbitrary Heegaard splitting: it is the lift of
the marked Seifert splitting.  This observation is what makes a uniform
standardness theorem possible.

The proof also clarifies two points of notation.  First, the name
\emph{Olshanskii epimorphism} should refer to the restricted product
epimorphism \eqref{eq:beta}, rather than to the coordinate-surjective
parent homomorphism.  Second, the word $V$ in the original table is a
basepoint path.  It is essential in that based table, but it is absent from
the unbased attaching curve, the balanced presentation, and the covering
quotient.

For the quaternion cover, the explicit calculation shows how the two
classical approaches meet.  The maximal tree gives the covering handlebody,
a geometric longitude--meridian pair on each lifted handle, and the two
Schreier bases.  Explicit products of Dehn twists supported on the nine
one-holed tori transform the lifted meridian words into a free basis, so the
basis criterion proves standardness.  The separate Andrews--Curtis
calculation reduces the associated balanced presentation.  In higher-order
examples these direct calculations become large, while the Seifert-cover
proof remains unchanged.

\section*{Acknowledgments}
We have used Chat GPT to edit and generate portions of the text and to make figures,  but we take full intellectual responsibility for the content. We used Magma for computations in Section 7.
The first author thanks the Dolciani--Halloran Foundation for its support.
Both authors thank PSC--CUNY for its support.

\end{document}